\documentclass[11pt, reqno]{amsart}

\usepackage{amssymb, amsmath, amsthm,mathtools}
\usepackage{hyperref}
\usepackage[alphabetic,lite]{amsrefs}
\usepackage{amscd}   % for commutative diagrams
\usepackage{fullpage}
\usepackage[all]{xy} % for complicated commutative diagrams
\usepackage{faktor}
\usepackage{setspace}
\usepackage{enumitem}

\DeclareFontEncoding{OT2}{}{} % to enable usage of cyrillic fonts

\usepackage[usenames,dvipsnames]{color}
\usepackage{tikz-cd}

\newtheorem{lemma}{Lemma}[section]
\newtheorem{theorem}[lemma]{Theorem}
\newtheorem{proposition}[lemma]{Proposition}

\newtheorem{conj}[lemma]{Conjecture}

\newtheorem{claim*}{Claim}

\theoremstyle{definition}
\newtheorem{remark}[lemma]{Remark}

\newcommand{\PP}{{\mathbb P}}
\newcommand{\C}{{\mathbb C}}
\newcommand{\F}{{\mathbb F}}
\newcommand{\Q}{{\mathbb Q}}

\newcommand{\Z}{{\mathbb Z}}
\newcommand{\NN}{{\mathbb N}}

\newcommand{\kk}{{\mathbf k}}

\usepackage[mathscr]{euscript}

\DeclareMathOperator{\im}{im}

\DeclareMathOperator{\Gal}{Gal}

\DeclareMathOperator{\Jac}{Jac}

\DeclareMathOperator{\SL}{SL}

\DeclareMathOperator{\GL}{GL}

\DeclareMathOperator{\gon}{gon}

\newcommand{\into}{\hookrightarrow}

\numberwithin{equation}{section}
\numberwithin{table}{section}

\title{Isolated Points on $X_1(\ell^n)$ with rational $j$-invariant}

\author{\"Ozlem Ejder}
\address{Bo\u{g}azi\c{c}i  University, Department of Mathematics, Istanbul, Turkey}
\email{ozlem.ejder@boun.edu.tr}
\urladdr{https://sites.google.com/site/ozheidi/Home}

\begin{document}

\begin{abstract}
Let $\ell$ be a prime and let $n\geq 1$. In this note we show that if there is a non-cuspidal, non-CM isolated point $x$ with a rational $j$-invariant on the modular curve $X_1(\ell^n)$, then $\ell=37$ and the $j$-invariant of $x$ is either $7\cdot11^3$ or $-7.137^3\cdot2083^3$. The reverse implication holds
for the first j-invariant but it is currently unknown whether or not it holds for the second.
\end{abstract}

\maketitle

\section{Introduction}

Let $C$ be a curve over a field $k$. Frey \cite{Frey} observed that Faltings's theorem implies that if $C$ has infinitely many degree $d$ points, then either there is a function $C \to \PP^1$ of degree $d$ or that the image of the map $\phi_d: C^{(d)} \to \Jac(C)$ is a union of translates of a positive rank subabelian variety. We call a closed point on a curve $C$ isolated if it is neither a member of a family parametrized by $\PP^1$ or by a positive rank subabelian variety of the Jacobian of $C$. See the next page for a more precise definition. Motivated by the classification of torsion subgroups of elliptic curves over various number fields, we study the isolated points on $X_1(n)$. In \cite[Corollary 1.7]{BELOV},  it is proven that there are only finitely many  rational $j$-invariants giving rise to isolated points assuming Serre's uniformity conjecture (originally a question of Serre \cite{Serre}, formalized as a conjecture by Zywina \cite[Conj 1.12]{Zywina-possibleimages} and Sutherland \cite[Conj 1.1]{Sutherland-Computingimages}). 

\begin{conj}[Uniformity conjecture]
For all non-CM elliptic curves $E/\Q$, the mod-$\ell$ Galois representation of
$E$ is surjective for all $\ell > 37$. 
\end{conj}

 In this short note, we prove unconditionally that there are finitely many isolated rational $j$-invariants on $X_1(\ell^n)$ for any prime $\ell >7$.  
  \begin{theorem}\label{thm:primepower}
      Let $\ell$ be a prime greater than $7$ and let $n$ be a positive integer. If $X_1(\ell^n)$ has a non-CM, non-cuspidal isolated point with a rational $j$-invariant, then $\ell=37$ and the $j$-invariant is either $7\cdot11^3$ or $-7\cdot137^3\cdot2083^3$.
    \end{theorem}
    \begin{remark}
The first $j$-invariant $7\cdot11^3$ gives rise to an isolated point on $X_1(37)$ by \cite[Proposition 8.4]{BELOV}. However, we currently do not know whether the second $j$-invariant $-7\cdot137^3\cdot2083^3$ gives rise to an isolated point on $X_1(37)$ or not. We also note that the case $\ell=2$ was studied in \cite[Theorem~8.5]{BELOV} (for sporadic points) and it is an open problem to determine the isolated rational $j$-invariants on $X_1(\ell^n)$ for $\ell=3,5,7$ and $n\geq 2$.
\end{remark}

Another unconditional result related to this problem is given in \cite{BRRWodd} for isolated points of odd degree with rational $j$-invariant on $X_1(n)$.
Also see \cite[Theorem~2.3]{smith2020ramification} for a result on sporadic points on $X_1(\ell^n)$ corresponding to elliptic curves with supersingular reduction at $\ell$.

\section{Background and Notation}
 By curve we mean a projective nonsingular $1$-dimensional scheme over a field.  For a curve $C$ over a number field $k$, we use $\gon_k(C)$ to denote the $k$-gonality of $C$, which is the minimum degree of a dominant morphism $C \rightarrow \mathbb{P}^1_k$. If $x$ is a closed point of $C$, we denote the residue field of $x$ by $\kk(x)$ and define the degree of $x$ to be the degree of the residue field $\kk(x)$ over $k$. 
 
 If $E$ is an elliptic curve defined over a number field $k$ and $P \in E(k)$, then $k(P)$ denotes the field extension of $k$ generated by the $x-$ and $y-$coordinates of $P$.

  We use $E$ to denote an elliptic curve, i.e., a curve of genus $1$ with a specified rational point $O$. Throughout we will consider only elliptic curves defined over number fields. We say that an elliptic curve $E$ over a field $k$ has complex multiplication, or {CM}, if the geometric endomorphism ring is strictly larger than $\mathbb{Z}$.

 \subsection{Galois Representations}
Let $k$ be a number field. Throughout, we denote the absolute Galois group of $k$ by $G_{k}$. We use $\ell$ to denote an odd prime number. 
Let $E/k$ be an elliptic curve defined over the number field $k$. Fixing a basis for the $\ell$-adic Tate module of $E$, we obtain the representation given by  
\[ \rho_{E,\ell^{\infty}}: G_{k} \to \GL_2(\Z_{\ell}),
\]
 where $\Z_{\ell}$ denotes the ring of $\ell$-adic integers. Similarly for any $n \geq 1$, we also have 
 \[ \rho_{E,\ell^n}: G_{k} \to \GL_2(\Z/{\ell^n}\Z),
\]
which describes the action of $G_k$ on the $\ell^n$-torsion subgroup $E[\ell^n]$ of $E(\bar{k})$. We note that $ \rho_{E,\ell^n}=\pi_n \circ  \rho_{E,\ell^{\infty}}$ where $\pi_n$ is the natural projection map $ \GL_2(\Z_{\ell}) \to  \GL_2(\Z/\ell^n \Z)$. 

We denote the image of $ \rho_{E,\ell^n}$ (resp., $ \rho_{E,\ell^{\infty}})$ as $G_{E,\ell^n}$ (resp., $G_{E,\ell^{\infty}}$).

  \subsection{Modular Curve $X_1(n)$}

  For a positive integer $n$, let
        \[
            \Gamma_1(n) \coloneqq
            \left\{{\left(\begin{smallmatrix}a & b \\ c & d \end{smallmatrix}\right)} \in \SL_2(\Z) :  c \equiv 0 \pmod{n}, \, a \equiv d \equiv 1 \pmod{n}\right\}. 
        \]
        The group $\Gamma_1(n)$ acts on the upper half plane $\mathbb{H}$ via linear fractional transformations, and the points of the Riemann surface $Y_1(n) \coloneqq \mathbb{H}/\Gamma_1(n)$ correspond to equivalence classes of pairs $[(E,P)]$, where $E$ is an elliptic curve over $\C$ and $P\in E$ is a point of order $n$. Here two pairs $(E,P)$ and $(E',P')$ 
     are equivalent if there exists an isomorphism $\varphi\colon E \rightarrow E'$ such that $\varphi(P)=P'$. By adjoining a finite number of cusps to $Y_1(n)$, we obtain the smooth projective curve $X_1(n)$. In fact, we may view $X_1(n)$ as an algebraic curve defined over $\Q$. 

 \begin{lemma}\cite[Lemma 2.1]{BELOV}\label{lem:degree}
            Let $E$ be a non-CM elliptic curve defined over the number field $k=\Q(j(E))$, let $P\in E$ be a point of order $n$, and let $x = [(E,P)]\in X_1(n)$. Then
            \[
                \deg(x)=c_x[k(P):k],
            \]
            where $c_x=1/2$ if {$2P \neq O$ and} there exists $\sigma \in \Gal_k$ such that {$\sigma(P)=-P$} and $c_x=1$ otherwise.
    \end{lemma}
   
    \begin{proposition}\cite[Proposition 2.2]{BELOV}\label{prop:degree}
                For positive integers $n \geq m$ and a prime $\ell$, there is a natural $\Q$-rational map $\pi \colon X_1(\ell^n) \rightarrow X_1(\ell^m)$ with
                \[
                    \deg(\pi)=\ell^{2(n-m)}                \]
          \end{proposition}

\subsection{Isolated Points}

 Let $C/k$ be a curve with a point $P \in C(k)$. For $d\in \NN$, let $C^d$ denote the direct product of $d$ copies of $C$. We denote the $d$-th symmetric product of $C$ by $C^{(d)}$, i.e., the quotient of $C^d$ by the symmetric group $S_d$.  We have a natural map
 $\phi_d \colon C^{(d)} \to \Jac(C)$ given by $(P_1,\ldots, P_d) \mapsto [P_1+\ldots P_d-dP]$. We say that a point $x$ on $C$ is isolated \cite[Definition 4.1]{BELOV} if 
 \begin{itemize} 
 \item($\PP^1$-isolated) there is no $x' \in C^{(d)}$ such that $\phi_d(x)=\phi_d(x')$ and
 \item (AV-isolated)  there is no positive rank subabelian variety $A$ of $\Jac(C)(k)$ such that $\phi_d(x) + A \subset \im(\phi_d)$.
 \end{itemize}
 The first condition is due to the fact that if such a point exists, then there has to be a rational map $f \colon C \to \PP^1$ of degree $d$ such that $x$ is in $f^{-1}(\PP^1(k))$. If this is the case, we say $x$ is a member of a family  parametrized by $\PP^1$. Similarly, if there is such an abelian variety, we say $x$ is parametrized by a positive rank subabelian variety of $\Jac(C)$.
We note here that if $\Jac(C)(\Q)$ is of rank zero and the degree of a point $x$ on $C$ is less than the gonality, then $x$ is isolated.
   
    \begin{lemma}\label{lem:RiemannRoch}
    Let $C/k$ be a curve of genus $g>0$. Let $x$ be a point on $C$ of degree $d$. If $x$ is an isolated point, then $d \leq g$.
     \end{lemma}
     \begin{proof}
     Assume that $d >g$. Let $x$ be a point of degree $d$ on $C$. Then $D=\sum_i{x_i}$, where $x_i$ are Galois conjugates of $x$, is a degree $d$ divisor. By the Riemann-Roch theorem, $\ell(D) \geq d-g+1 \geq 2$ and hence there is a non-constant function $f:C \to \PP^1$ defined over $k$ whose poles are at most at $x_i$'s. Since $f$ is defined over $k$, if $x_j$ is a pole of $f$, then $x_i$ is a pole of $f$ for all $i$. We deduce that $f$ has degree $d$ which implies that $x$ is not $\PP^1$-isolated, hence it is not isolated. \end{proof}
    
    \begin{theorem}\cite[Theorem 4.3]{BELOV}\label{levellowering}
     Let $f : C \to D$ be a finite map of
curves and let $x\in C$ be an isolated point. If $\deg(x) = \deg(f(x)) \dot \deg(f)$, then $f(x)$ is an isolated
point of $D$.
\end{theorem}

\begin{remark} \label{rem:level} 
Let $\pi: X_1(\ell^n) \to X_1(\ell^m)$ for integers $n> m$. Let $x:=[(E,P)]$ be a point on $X_1(\ell^n)$. If $G_{E,\ell^{n}}=\pi^{-1}(G_{E,\ell^{m}})$, then the assumption of Theorem \ref{levellowering} holds by \cite[Corollary 5.3]{BELOV}. This holds in particular when $\ell >3$ and $\rho_\ell$ is surjective.
\end{remark}

We call a point  $j\in X_1(1)\simeq \PP^1$ an isolated $j$-invariant if it is the image of an isolated point on $X_1(n)$, for some positive integer $n$.

\subsection{Some Subgroups of $\GL_2(\Z/\ell\Z)$} 
The nonsplit Cartan subgroup of $\GL_2(\Z/\ell \Z)$ is the subgroup
\[   C_{ns}(\ell)=\left\{ 
	\begin{bmatrix} 
	               a &\epsilon b \\
	               b & a\end{bmatrix}:a,b \in  \Z/\ell\Z,   (a,b)  \not\equiv (0,0)  \pmod \ell 
	               \right\}
 \]
where $\epsilon$ is a non-quadratic residue modulo $\ell$. We denote the normalizer of $C_{ns}(\ell)$ by $C_{ns}^+(\ell)$ respectively. The group $C_{ns}(\ell)$ has order $\ell^2-1$ and $C_{ns}^+(\ell)$ has order $2(\ell^2-1)$.

Let $E$ be an elliptic curve defined over $\Q$ and let $\ell \geq 5$ be a prime. Let $K$ be an
extension of $\Q_\ell$, of the least possible degree such that $E/K$ has good or multiplicative
reduction. Let $e$ be the ramification index of $K/\Q_\ell$.  Let $D$ denote the semi-Cartan subgroup of $\GL_2(\Z/\ell \Z)$ given by
 \[D=\left\{ 
	\begin{bmatrix} 
	               a &0 \\
	               0& 1\end{bmatrix}: a \in {\Z/\ell \Z}^* 
	              \right \}.
 \]
Here $e$ is $1,2,3,4$ or $6$. Let $f=\gcd(\ell-1,e)$, then $f <5$ or $f=6$.
\begin{theorem}\cite{Serre}\label{Serre-results}
 If $E/K$ has potential good ordinary or multiplicative reduction at $p$, then $G_{E,\ell}$ contains a subgroup that is conjugate to $D^f$.
\end{theorem}
\begin{proof} See \cite[Theorem~3.1]{LR-TorsionFieldOfDefn}. 
\end{proof}

Let $E/\Q$ be a non-CM elliptic curve. Then $G_{E,\ell}$ is either $\GL_2(\Z/\ell\Z)$ or it is contained in one of the maximal subgroups of $\GL_2(\Z/\ell \Z)$: the normalizer of Cartan subgroups, Borel subgroups and the exceptional subgroups. Mazur \cite{Mazur} showed that if it is contained in a Borel subgroup, then $\ell$ is in $ \{2,3,5,7,11,17,37\}$. Moreover, if $\ell$ is $17$ or $37$, then $j(E)$ is in 
  \[ \{ -{17}\cdot{{373}^3}/{2^{17}}, {-17^2\cdot101^3}/2, -7\cdot11^3, -7\cdot137^3\cdot 2083^3 \}.
  \]
 See \cite{Zywina-possibleimages}.
In the case of the normalizer of a split Cartan subgroup, by \cite{BP11} and \cite{BPR13}, we know that $\ell \leq 7$ or $\ell=13$. Recent progress on finding rational points on curves \cite{BDMTV} showed that $\ell$ cannot be $13$. Similarly, Serre himself showed that if the group $G_{E,\ell}$ is contained in an exceptional subgroup of $\GL_2(\Z/\ell \Z)$,  then $\ell$ must be less than or equal to $13$. Hence if $\ell \geq 17$ and $\rho_{E,\ell}$ is not surjective, then it is either contained in the normalizer of a nonsplit Cartan subgroup of $\GL_2(\Z/\ell \Z)$ or the $j$-invariant is in the list $\{ -{17}\cdot{{373}^3}/{2^{17}}, {-17^2\cdot101^3}/2, -7\cdot11^3, -7\cdot137^3\cdot 2083^3 \}$. 

In the Borel case, we know that $G_{E,\ell^{\infty}}$ is as large as possible given the group $G_{E,\ell}$ for $\ell \leq 7$. Although Greenberg proves a similar result also for $\ell=5$, we only need to use the case $\ell>5$ in this article.

\begin{theorem}\cite{Greenberg-isogeny}, \cite{greenbergRSS:7isogeny}\label{Greenberg-isogeny}
Assume that $\ell > 5$. Assume that $E/\Q$ is a non-CM curve with an $\ell$-isogeny. Then the image of $\rho_{E,\ell^{\infty}}$ contains a Sylow pro-$\ell$ subgroup of $\GL_2(\Z_\ell)$.
\end{theorem}

\section{Classifying isolated Points on Prime Power Level}

Let $E$ be an elliptic curve over $\Q$. Then \cite[Proposition 2.2]{Lemos-borel} implies that if $\ell \geq 5$ and $G_{E,\ell}$ is contained in $C_{ns}^+(\ell)$, then $E$ has potential good reduction at $\ell$. We show that $E$, in fact, has potential good supersingular reduction at $\ell$ when $\ell >7$.

 \begin{proposition} \label{supersingular}
	Assume that $\ell>7$ and $\ell \neq 13$. Let $E$ be an elliptic curve defined over $\Q$. If $G_{E,\ell}$ is conjugate to a subgroup of  $C^{+}_{ns}(\ell)$, then $E$ has potential supersingular reduction at $\ell$.
	\end{proposition}
	
	\begin{proof}
	Let $\ell >7$. Fixing a basis for $E[\ell]$, we may assume that $G_{E,\ell}$ is contained in  $C^{+}_{ns}(\ell)$. Assume for contradiction  that $E$ has potential good ordinary or multiplicative reduction at $\ell$. By Theorem \ref{Serre-results}, $G_{E,\ell}$ contains a subgroup $H$ that is conjugate to $D^f$, the $f$'th power of a semi-Cartan subgroup.	
		
 We first consider the composition of the inclusion map $H \into C_{ns}^{+}(\ell)$ and the quotient map $ C_{ns}^{+}(\ell) \to C_{ns}^{+}(\ell)/C_{ns}(\ell)$. We observe that the kernel of this composition is $H\cap C_{ns}(\ell)$ and hence we have an injective map
   \[ 
          H/H\cap C_{ns}(\ell) \into C_{ns}^{+}(\ell)/C_{ns}(\ell). 
    \]
 Since the order of $C_{ns}^{+}(\ell)/C_{ns}(\ell)$ is two, the index of the subgroup $H\cap C_{ns}(\ell)$ in $H$ is at most $2$. We also note that the order of $H$ (and also the order of $H\cap C_{ns}(\ell)$) divides $\ell-1$.
 The group $C_{ns}(\ell)$ is isomorphic to ${\F}_{\ell^2}^{*}$ by the map 
	\[ 
	\begin{bmatrix} 
	   a & b\epsilon \\
	   b & a
	\end{bmatrix} \mapsto a+\epsilon b,
	\]
where $\epsilon$ is a non-quadratic residue modulo $\ell$ and hence it is cyclic. The unique subgroup of $C_{ns}(\ell)$ of order equal to $|H\cap C_{ns}(\ell)|$ is isomorphic to a subgroup of ${\F}_{\ell}^{*}$, i.e., it is isomorphic to a subgroup of the group of diagonal matrices. 
	%$\{ \begin{bmatrix} a & 0\\
	%0 & a\end{bmatrix} : a \in {\F}_{\ell}^{*}\}.$
	
A matrix in $D^f$ has two eigenvalues: $1$ and $a$. However a diagonal matrix has one eigenvalue with multiplicity two. Hence $H\cap C_{ns}(\ell)=\{ (1)\}$ and $H$ has at most two elements. For $\ell>13$, the order of $H$ which equals $(\ell-1)/f$ is strictly greater than $2$ since $f\leq 6$. This proves that $E$ has potential supersingular reduction at $\ell$ for $\ell>7$ and $\ell\neq 13$. 
\end{proof}

\begin{proposition}\label{prop:nonsplit}
Let $\ell >7$ and $\ell \neq 13$. Let $E$ be an elliptic curve defined over $\Q$ such that the image of $\rho_{E,\ell}$ is conjugate to a subgroup of $C_{ns}^+(\ell)$. If $R$ is a point of order $\ell^n$  on $E$, then the degree of $\Q(R)$ over $\Q([\ell]R)$ equals $\ell^2$.
\end{proposition}
\begin{proof}
By Proposition~\ref{supersingular}, the elliptic curve $E$ has potential supersingular reduction at $\ell$. Let  $R$ be a point of exact order $\ell^n$ on $E$. By \cite[Theorem~1.2(2)]{Lozano-Robledo-supersing}  the degree of the extension $\Q(R)$ over $\Q([\ell]R)$ is divisible by $\ell^2$.  Since the degree $[\Q(R):\Q([\ell]R)]$ can be at most $\ell^2$, we are done.
\end{proof}
\begin{lemma}\label{lem:nonsplit}
Let $\ell>7$ and $\ell \neq 13$. Let $x=[(E,P)]$ be a point on $X_1(\ell^n)$ such that $G_{E,\ell}$ is conjugate to a subgroup of  $C_{ns}^+(\ell)$. Then $ \deg(x) = \deg(\pi(x)) \dot \deg(\pi)$
where $\pi : X_1(\ell^n)\to X_1(\ell^m)$ for any $n>m$.

\end{lemma}
\begin{proof}
This follows from Lemma~\ref{lem:degree}, Proposition~\ref{prop:degree} and Proposition~\ref{prop:nonsplit}.
\end{proof}

\begin{remark}\label{nonsplit13}
 If $\ell=13$, then there are no non-CM elliptic curves with Galois representation contained in $C_{ns}^{+}(\ell)$ (\cite{BDMTV}). Hence the conclusion of Proposition~\ref{supersingular}, Proposition~\ref{prop:nonsplit} and Lemma~\ref{lem:nonsplit} holds for $\ell >7$ and for all non-CM elliptic curves defined over $\Q$.
\end{remark}

   \subsection{Proof of Theorem \ref{thm:primepower} }

Let $x=[(E,P)]$ be a non-CM, non-cuspidal isolated point on $X_1(\ell^n)$ with a rational $j$-invariant. We may assume that $E$ is defined over $\Q$. For $\ell >7$ and $\ell \neq 13$, $\rho_{E,\ell}$ is either surjective, contained in a Borel subgroup, or the normalizer of a non-split Cartan subgroup. When $\ell=13$, it can also be contained in an exceptional subgroup of $\GL_2(\Z/\ell\Z)$ by the classification given in \cite{Zywina-possibleimages}. We will first show that when $\ell>7$, $x$ induces an isolated point on $X_1(\ell)$. Then we will rule out the existence of an isolated point on $X_1(\ell)$.

Assume $\ell >7$. If $G_{E,\ell}$ is contained in a Borel subgroup, then by Theorem~\ref{Greenberg-isogeny}, the $\ell$-adic representation of $E$ is as large as possible given the mod $\ell$ representation. Using Remark \ref{rem:level} and Theorem \ref{levellowering} we conclude that if the image of $\rho_{E,\ell}$ is $\GL_2(\Z/\ell\Z)$ or it is contained in a Borel subgroup, then $x$ maps to an isolated point on $X_1(\ell)$.

We assume now that $G_{E,\ell}$ is contained in the normalizer of a non-split Cartan subgroup. In this case, we do not know that the $\ell$-adic representation is determined by the mod~$\ell$ image. However, by Lemma~\ref{lem:nonsplit} and Theorem~\ref{levellowering}, we are able to conclude that the image of $x$ on $X_1(\ell)$ is isolated.

 Assume that mod $\ell$ representation $G_{E,\ell}$ is exceptional. By the classification of the images of $\rho_{E,\ell}$ given in \cite{Zywina-possibleimages}, we may assume $\ell=13$. Moreover by the results of \cite{BDMTV}, we know the (finitely many) $j$-invariants giving rise to these points. Recent work  \cite{rouse2021elladic} of Rouse, Sutherland and Zureick-Brown shows that $\ell$-adic representation of $E$ in this case is as large as possible given the mod $\ell$ representation. By Remark \ref{rem:level}, $x$ induces an isolated point on $X_1(\ell)$.

We may now assume that $x$ is an isolated point on $X_1(\ell)$ with a rational $j$-invariant. The rest of the proof is similar to the proof of \cite[Proposition~8.4]{BELOV}. The genus of $X_1(\ell)$ is less than $(\ell^2-1)/24$ for prime $\ell$. If the image of $\rho_{E,\ell}$ is contained in the normalizer of a nonsplit Cartan subgroup, then the degree of $x$ is at least $(\ell^2-1)/12$ by \cite[Theorem 7.3 ]{LR-TorsionFieldOfDefn}. By Lemma~\ref{lem:RiemannRoch}, $x$ is not isolated. Assume $G_{E,\ell}$ is contained in a Borel subgroup. Then $\ell=11, 17$ or $37$. Assume $\ell=11$. Since $X_1(11)$ has genus one, $x$ cannot be isolated by Lemma~\ref{lem:RiemannRoch}. 
Assume $\ell=17$. Then the degree of $x$ is either $4$ or $8$. By \cite[Proposition~6]{DKM}, there are no $\PP^1$-isolated points of degree $4$. Since the Jacobian of $X_1(17)$ has only finitely many rational points, it follows that there are no isolated points of degree $4$ on $X_1(17)$. Since the genus is $5$, a degree $8$ point cannot be isolated by Lemma \ref{lem:RiemannRoch}. On the other hand, there are two rational $j$-invariants giving rise to an elliptic curve with a rational $37$-isogeny. They are given by $7\cdot11^3$ and $-7\cdot137^3\cdot2083^3$. The first one is isolated by (\cite[Proposition 8.4]{BELOV}). %however we do not currently know if the second one is isolated or not. 

Let $\ell=13$. We have covered all cases except the exceptional subgroup. There are three such rational $j$-invariants. We compute using Magma that the degree of these points on $X_1(13)$ are greater than $3$, since the genus is $2$, we are done.
\qed

\subsection*{Acknowledgements}
The author is grateful to Abbey Bourdon, Filip Najman, Alvaro Lozano-Robledo and the anonymous referee for their comments on the earlier drafts of this paper. The author is supported by the project Marie Skłodowska-Curie actions and TUBITAK.

%\bibliography{mybib}{}
%\bibliographystyle{plain}

\end{document}